\documentclass[12pt, a4paper]{report}
\usepackage{amsmath, amsthm, amssymb,amscd}
\usepackage{cite}
\usepackage{float}
\usepackage{graphicx}
\usepackage{setspace}
\usepackage[all,cmtip]{xy}
\usepackage[english]{babel}
\usepackage{appendix}
\newcommand{\proj}{\mathbb{P}}

\newcommand{\seq}{\subseteq}
\newcommand{\C}{\mathbb{C}}
\newcommand{\R}{\mathbb{R}}

\newtheorem{thm}{Theorem}[section]
\newtheorem{mydef}[thm]{Definition}
\newtheorem{lem}[thm]{Lemma}

\newtheorem{cor}[thm]{Corollary}
\newtheorem{conjecture}[thm]{Conjecture}
\newtheorem{prop}[thm]{Proposition}

\theoremstyle{remark}
\newtheorem*{rem}{Remark}

\newtheorem{eg}{Example}

\title{The Universal Severi Variety of Rational Curves on K3 Surfaces}
\author{Michael Kemeny}

\begin{document}
\maketitle
\tableofcontents
\chapter*{Introduction}
In this paper we study the universal Severi variety of irreducible curves of a fixed geometric genus $h$ and in a fixed linear system on 
a primitively polarized K3 surface of genus $g$, with $g>2$ assumed throughout. This work
is an edited version of my Master's Thesis ``Rational Curves on K3 Surfaces'' at the Universit\"{a}t Bonn, Germany.

Let $S$ be a projective surface over $\C$ and let $M \in Pic(S)$ be 
an effective line bundle. Then let $V_h(S,M)$ be the fine moduli space parametrising irreducible, nodal curves 
of geometric genus $h$ in the linear system $|M|$. We call $V_h(S,M)$ the \emph{Severi variety of genus $h$ curves in $S$}.

The spaces $V_h(S,M)$ are classical objects, which were first studied by Severi in the case $S=\proj^2$.
There are two natural questions which immediately arise. Firstly, one would like to know whether the space $V_h(S,M)$
is smooth. Secondly, one may ask whether the space $V_h(S,M)$ is irreducible. 

The first question is in general easier than the second. For example the space $V_h(S,M)$ is known to be smooth
if $S$ is either the plane $\proj^2$, a del Pezzo
surface, a K3 surface or an Enriques surface. There are, however,
examples of surfaces $S \seq \proj^3$ where $V_h(S,M)$ is \emph{not} smooth, see \cite{chian}. 

The second question is harder. The case $S=\proj^2$ was first stated by Severi, who famously gave an
incorrect proof that $V_h(\proj^2, M)$ was irreducible. A correct proof was given by Harris \cite{harrissev}.
The case where $S$ is a del Pezzo surface and $h=0$ was dealt with by Testa in his doctoral thesis \cite{testa},
where it is shown that $V_0(S, M)$ is empty or irreducible so long as the degree of the del Pezzo surface is not
one. In case the degree is one, the result holds if the del Pezzo surface is generic and $M$ is not the anticanonical bundle. 

When $S$ is a generic K3 surface with Picard group $Pic(S) \simeq \mathbb{Z}$ and $M \in Pic(S)$ effective and primitive, 
then the space $V_0(S, M)$ consists of finitely many points, the number of which
is given by the Yau--Zaslow formula \cite{beau-97}. In particular, the Severi variety of rational curves in $S$ is \emph{not} irreducible. On the other hand,
if $h \neq 0$, then little is known about the irreducibility of $V_h(S, M)$ with $h \neq 0$ and
$S$ a generic K3 surface. If the geometric genus $h$ of the curves is sufficiently high with respect to the genus $g$ of the K3 surface, then
the irreducibility of the Severi variety can be shown. In Appendix \ref{sect:hilb} we will use the Hilbert scheme of points and the notion of a $k$-very ample
line bundle to prove that $V_h(S, M)$ is irreducible for $\frac{5g-1}{6} \leq h \leq g$ (with $S$ a generic K3 surface and $M$ primitive). 

Instead of working with one fixed K3 surface, we may consider the moduli stack $\mathcal{B}_g$ of primitively polarized
K3 surfaces of genus $g$, with $g>2$. So then $\mathcal{B}_g$ parametrises all pairs $(S,M)$ with $S$ a K3 surface and $M \in Pic(S)$
an ample and primitive line bundle such that $(M \cdot M)= 2g-2$. Let $\mathcal{U} \seq \mathcal{B}_g$ denote the open substack of polarized
K3 surfaces $(S,M)$ such that $V_h(S,M)$ nonempty.
Denote by $\mathcal{V}^g_{h,k}$ the stack over $\mathcal{U}$ with fibre over $(S,M)$
given by $V_h(S,kM)$. We call this space the universal Severi variety of genus $h$ curves in the linear system $|kM|$ on primitively polarized K3 surfaces $(S,M)$ of 
genus $g$.

The space $\mathcal{V}^g_h(S,kM)$ is studied in \cite{dedieu}, where it appears in relation to the study of self-rational
maps of K3 surfaces. Once again, we may ask whether or not this space is smooth and/or irreducible.
The first question is answered in the positive in \cite{flam}. Regarding the second question, the 
following conjecture has been made:
\begin{conjecture} \label{conj}
 The universal Severi variety $\mathcal{V}^g_{h,k}$ is irreducible.
\end{conjecture}
It was shown by Dedieu that this conjecture would imply that there are no self-rational maps $S \dashrightarrow S$ of
degree $ \geq 2$ on a \emph{generic} K3 surface $S$ with $Pic(S)= \mathbb{Z}$.  A proof of the conjecture for $g < 12$ appears in the recent paper \cite{cilided}. 
The nonexistence of self-rational maps of degree $\geq 2$ on generic K3 surfaces has since
been proven using a different approach by X. Chen, see \cite{chenrat}.


We will primarily concern ourselves with the \emph{rational} case, i.e. the case $h=0$. In this case 
the fibres $V_0(S,kM)$ are generically equal to finitely many isolated points, and hence they
 are not irreducible. It is still, however, perfectly possible that the universal space $\mathcal{V}^g_{0,k}$ is irreducible. 
 In fact, in Section \ref{sect:low} we will prove the irreducibility of $\mathcal{V}^3_{0,k}$, i.e. the case of rational curves on quartic surfaces, for the first few values of $k$.

Since $\mathcal{V}^g_{0,k}$ is smooth, it would be sufficient to prove that this space is connected,
in order to establish Conjecture \ref{conj} for the case $h=0$. Unfortunately, this is very
difficult to do, at least via explicit computations. This is due to the simple reason that one does
not have a good description of the open, dense substack $\mathcal{U} \seq \mathcal{B}_g$. 
Many of the simplest examples of polarised K3 surfaces are not contained in this stack; for instance,
an elliptic K3 surface with Picard number two polarised by a primitive, ample class does not live in this space. Indeed, any rational
curve in a primitive, ample class on such an elliptic K3 surface is \emph{reducible}. Since a standard approach
to dealing with the moduli space of K3 surfaces is to first prove the result on the dense subspace of elliptic surfaces and then
extend the result, this example is somewhat troubling.

In order to rectify this problem, we will in Section \ref{sect:main} drop the condition that the rational curves are irreducible and nodal.
We will consider the stack $|\mathcal{M}|$ over $\mathcal{B}_g$ with fibre over $(S,M)$
given by the complete linear system $|M|$, and look at the closure of $\mathcal{V}^g_{0,1}$ in
$|\mathcal{M}|$, which we denote by $\overline{\mathcal{V}}^g_{0,1}$. 
The advantage of this space is that we can now give many examples of points in 
$\overline{\mathcal{V}}^g_{0,1}$. In fact let $S \to \proj^1$ be an elliptic K3 surface with generic fibre $E$ and
section $T$. Suppose that $S$ has either 24 nodal singular fibres or 12 cuspidal singular fibres, and let $N$
be any singular fibre. Then the (nonreduced) curve $T+gN \in |T+gE|$ is rational, and further
one has $(T+gN,S,T+gE) \in \overline{\mathcal{V}}^g_{0,1}$.

Conjecture \ref{conj} for $h=0$ is then equivalent to the statement that $\overline{\mathcal{V}}_{0,1}$
is irreducible. The main result of this thesis is the following:
\begin{thm} \label{mainthm}
 The space $\overline{\mathcal{V}}^g_{0,1}$ is connected for all $g >2$. 
\end{thm}
Note that this does not resolve the conjecture completely, since the space
 $\overline{\mathcal{V}}^g_{0,1}$ might no longer be smooth. It is reasonable to hope that a more detailed study
of the scheme structure of $\mathcal{V}^g_{0,1}$ would establish irreducibility and thereby resolve the conjecture in this case.

The idea behind our proof is very simple. Let $(S,M)$ be a primitively polarized K3 surface of genus $g$ and $C \in |M|$ a rational
curve, and assume that $(C,S,M) \in \overline{\mathcal{V}}^g_{0,1}$. Then as a first step we will deform $(C,S,M)$ to a rational curve on an elliptic K3 surface with $12$ cuspidal
fibres. The second step is to show there exists a path in $\overline{\mathcal{V}}^g_{0,1}$ connecting any two rational curves on elliptic K3 surfaces with
$12$ cuspidal fibres. This is achieved by explicit constructions of the required paths.

\emph{Acknowledgements.} This work was funded by a Qualification Scholarship from Bonn 
International Graduate School in Mathematics (BIGS). I am most grateful to my supervisor, Professor Daniel Huybrechts,
who introduced me to the topic and generously spent many hours discussing the various intricacies of K3 surfaces with me.  

\chapter{The Universal Severity Variety of Rational Curves on K3 Surfaces} 
Before we go further, we need to recall the basic results on the construction of Severi varieties on rational curves. Let $S$ be a smooth 
K3 surface with a a primitive, ample line bundle $L$ such that $(L \cdot L)=2g-2$. Recall from \cite{tannen} that there
is a smooth variety $V_{h}(S,kL)$ parametrising nodal, irreducible curves with fixed geometric genus $h$ in the linear system $|kL|$. The variety $V_{h}(S,kL)$ is either 
empty or of dimension $h$. Note that any curve $C \seq S$ in $|L|$ has fixed arithmetic genus $1+k^2(g-1)$; hence if $C$ is nodal and irreducible then $C$ has geometric genus $h$ if and only
if it has $1+k^2(g-1)-h$ nodes.  So fixing the geometric genus is equivalent to fixing the number of nodes.

The above results can be relativised to deal with families of primitively polarised K3 surfaces, see for example \cite{flam}. Let $\mathcal{B}_g$ denote the Deligne--Mumford stack parametrising primitively polarised K3 surfaces of genus $g$. Then there is an algebraic  stack $\mathcal{V}_{h,k}^g \to \mathcal{B}_g$ 
 with fibre over $(X,L) \in \mathcal{B}_g$
  given by $V_h(X,kL)$. In the case of genus zero, the stack $\mathcal{V}_{h,k}^g$
  is Deligne--Mumford and the morphism $\mathcal{V}_{0,k}^g \to \mathcal{B}_g$ is smooth. 
 
We can avoid the use of stacks in the following way. Let $p(t)=(g-1)t^2+2$
and $N=P(3)-1$. There
is an open subscheme $W_g$ of a Hilbert variety parametrising pairs $(Z,L)$ with
$Z \seq \proj^N$ a closed subscheme and $L \in Pic(Z)$ a line bundle such that $(Z,L)$ is
a primitively polarised K3 surface of genus $g$ and $\mathcal{O}_{\proj^n}(1)|_{Z} \simeq L^3$, by \cite[Sect. 2.3]{andre} and \cite[Lem. 6.18]{olsson}. Further, $W_g$ is a smooth and irreducible variety, which comes equipped with a universal family
$g: \mathcal{X} \to W_g$ and a line bundle $\mathcal{L} \in Pic(\mathcal{X})$ such that for all $t \in W_g$, $(\mathcal{X}_t,\mathcal{L}_t)$ is a primitively
polarized K3 surface. There is, moreover, a variety $|k \mathcal{L}| \to W_g$ with fibre over $t \in T$ given by $|k \mathcal{L}_t|$, and there is a locally closed subscheme $M^g_{0,k}$
of $|k \mathcal{L}|$ parametrising irreducible, nodal curves of geometric genus zero. We have an obvious $PGL(N+1)$ action with finite stablizers on $W_g$ and on $M_{0,k}^g$ such that
the projection $M_{0,k}^g \to W_g$ is $PGL(N+1)$ equivariant.
The quotient stack $W_g / PGL(N+1)$ is then isomorphic to $\mathcal{B}_g$ and further $\mathcal{V}_{0,k}^g \simeq M_{0,k}^g / PGL(N+1) \to \mathcal{B}_g$. Moreover, the proof that the morphism of stacks $\mathcal{V}_{0,k}^g \to \mathcal{B}_g$ is smooth also shows the
smoothness of the projection morphism $M_{0,k}^g \to W_g$. Obviously, the connectedness or
irreducibility of $M_{0,k}^g$ would imply that $\mathcal{V}_{0,k}^g$ is connected, resp. irreducible.

\section{The moduli space of rational curves on quartic surfaces} \label{sect:low}
In this section we study the moduli space of nodal, rational curves on K3 surfaces with genus three. In this case the
K3 surfaces are quartic surfaces in $\proj^3$. We prove the first two cases completely
and obtain partial results for the third. The key result we use here is the fact that the Severi variety of rational curves
on a Del Pezzo variety is always irreducible (with one simple exception), see \cite{testa}.
 
Let $X \seq \proj^3$ be a smooth, quartic surface. The space $\mathcal{U}$ of such surfaces gives a dense, open subset of $|\mathcal{O}_{\proj^3}(4)|$. 
Let $L$ be an ample line bundle on $X$. If $X$ is generic, then $Pic(X) \simeq \mathbb{Z}$ and the ample line
bundle $L$ is of the form $\mathcal{O}_{\proj^3}(l)|_{X}$ by Max Noether's Theorem.
From now on we drop the $\proj^3$ from the notation and write $\mathcal{O}(m)$ for the line bundle $\mathcal{O}_{\proj^3}(m)$.
We are interested then in the irreducibility of the moduli space of pairs:
\[
 \mathcal{M}_l:= \{ (X,C)  \; | \; X \in \mathcal{U},\; C \in |\mathcal{O}_X(l)|  \; \text{a rational, irreducible, nodal curve} \}.
\]

Let $\mathcal{Q} \seq \mathcal{U} \times \proj^3$ be the universal smooth quartic hypersurface with
 projection $k: \mathcal{Q} \to \mathcal{U}$. We may then view $\mathcal{M}_l$ as a smooth subvariety of the projective bundle 
$\textbf{Proj} (\text{Sym}(k_* (\mathcal{O}_{\proj^3}(l)|_{\mathcal{Q}}))^*)$. This is a projective bundle over $\mathcal{U}$ which has fibre over $X$ 
isomorphic to
$\proj (H^0(\mathcal{O}_X (l)))$.
\begin{lem}
 Assume $l \leq 3$. Then there is an isomorphism 
 $$h: \mathbf{Proj}(\mathrm{Sym}(k_* (\mathcal{O}_{\proj^3}(l)|_{\mathcal{Q}}))^*)  \to \mathcal{U} \times |\mathcal{O}(l)| .$$
\end{lem}
 \begin{proof}
 Consider the short exact sequence on $\proj^3$
\[
 0 \to \mathcal{O}(-4) \otimes \mathcal{O}(l) \to \mathcal{O}(l) \to \mathcal{O}_X(l) \to 0 
\]
 where $X$ is a quartic surface. For $l \leq 3$, taking the associated long exact sequence of cohomology gives us an isomorphism
\[
 H^0(\mathcal{O}(l)) \simeq H^0(\mathcal{O}_X (l))
\]
induced by the restriction morphism. This lifts to the desired isomorphism $h$. To be precise, let $j: \mathcal{U} \times \proj^3 \to \proj^3$
be the projection. Then $h$ is induced by the isomorphism $ (k_* j^* \mathcal{O}_{\proj^3}(l)|_{\mathcal{Q}})^* \to (k_* j^* \mathcal{O}_{\proj^3}(l))^*$, 
since $\textbf{Proj} (\text{Sym}(k_* j^* \mathcal{O}_{\proj^3}(l))^*) \simeq   \mathcal{U} \times |\mathcal{O}(l)|$ (this last step comes about from the fact that
$k_* j^* \mathcal{O}_{\proj^3}(l)$ is the trivial bundle on $\mathcal{U}$ of rank equal to $\dim H^0(\proj^3,\mathcal{O}_{\proj^3}(l))$).
 \end{proof}

\begin{lem}
 Let $S \in |\mathcal{O}(l)|$ be an integral hypersurface and let $W_S$ denote the Severi variety of irreducible, rational, nodal curves
on $S$ in the linear system $|\mathcal{O}_S (4)|$. Assume $W_S$ is nonempty. Then $\dim(W_S)=11$ for $l=1$ and $\dim(W_S)=15$ for $l=2$.
\end{lem}
\begin{proof}
For $l=1$, $W_S$ is the 
Severi variety of irreducible, rational, nodal plane curves $\proj^1 \to \proj^2$ of degree four. So in this case
it is elementary to see that $W_S$ is irreducible and smooth with $\dim(W_S)=11$.

For $l=2$ we have two cases to consider. Firstly, there is the case where $S$ is smooth. Then $S$ is an irreducible, smooth, projective rational surface with
anti-ample canonical bundle $K_S=\mathcal{O}_S (-2)$. Hence by \cite{tannen2} we have $W_S$ is smooth of $\dim(W_S)=-(K_S \cdot \mathcal{O}(4))-1=15$. 
Next, suppose $S \seq \proj^3$ is an integral, but singular, quadric surface. From Jacobi's Theorem, the only quadratic forms in four variables are, up 
to projective equivalence,
\[
 0, \; x_0^2, \; x_0^2+x_1^2, \; x_0^2+x_1^2+x_2^2, \; x_0^2+x_1^2+x_2^2+x_3^3.
\]
The only forms from this list which define irreducible surfaces are $x_0^2+x_1^2+x_2^2$ and $x_0^2+x_1^2+x_2^2+x_3^2$. The latter
defines a smooth quadric surface, whereas the first of these defines a conic quadric with a singular point at the vertex $v=[0:0:0:1]$. So we assume
that $S$ is such a conic quadric. Now, following \cite[p.\ 374, ex.\ 2.11.4]{har}, blowing up at $v$ produces a smooth, rational ruled surface $S'$ over $\proj^1$
and a birational map $f: S' \to S$ which is an isomorphism outside of $f^{-1}(v)=E$ and with $E$ a smooth, rational section $\proj^1 \to S'$ satisfying
$(E \cdot E)=-2$. The Severi variety $W_S$ of irreducible, rational, nodal curves with associated line bundle $L:=\mathcal{O}_S (4)$ can be stratified as $W_S=W_1 \sqcup W_2$
with $W_1$ the closed subset of irreducible, rational curves in $|L|$ which pass through $v$ and $W_2=W_S - W_1$. We claim that $W_i$ is either empty or has dimension
$15$ for $i=1,2$, from which it follows that $\dim(W)=15$.

We start by identifying $W_1$ and $W_2$ with Severi varieties on $S'$. If $C \in W_1$, then $f^{-1}(C)$ has the form $C'+mE$ for some positive integer $m >0$ and with $C' \in |f^*(L)-mE|$ irreducible and rational.
Let $Y_m$ denote the space of irreducible, rational curves in $|f^{*}(L)-mE|$. We can then identify $W_1$ with the countable union $\displaystyle \coprod_{m \geq 1} Y_m$. Likewise, we can identify $W_2$ with the space of irreducible, rational curves in $|f^*(L)|$. Let $F \seq S'$ be 
a fibre of the ruling $S' \to \proj^1$; we then have that $E, F$ generate $Pic(S')$ and $(F \cdot F)=0$, $(F \cdot E)=1$. The canonical bundle of $S'$ is given by
$K_{S'}=-2E-4F$ from \cite[p.\ 374, Cor.\ 2.11]{har}. The obvious relations $(f^*(\mathcal{O}(1)) \cdot f^*(\mathcal{O}(1)))=2$, $(f^*(\mathcal{O}(1)) \cdot E)=0$ readily give
that $f^*(\mathcal{O}(1))=2F+E$ and so $f^*(L)=4(2F+E)$. Noting that $(f^{*}(L)-mE \cdot F)=4-m<0$ for $m >4$, we actually have that
$Y_m$ is empty for $m >4$ and so $W_1$ is a finite union $\displaystyle \coprod_{1 \leq m \leq 4} Y_m$. 
For any rational surface $Z$ and line bundle $T \in Pic(Z)$ with $(T \cdot K_F) < 0$, the Severi variety of irreducible, rational curves
in $|T|$ is either empty or of dimension $-(T \cdot K_F)-1$, see \cite{tannen2}. 
 We calculate $-(f^*(L)-mE \cdot K_{S'})=16=-(f^*(L) \cdot K_{S'})$ and hence $\dim(Y_i)=15$ and $\dim(W_2)=15$, so long as $Y_i$, resp. $W_2$, are nonempty. 
 Thus $W_S$ has dimension $15$.

\end{proof}

Now consider the morphism $p:=pr_2 \circ h^{-1}|_{\mathcal{M}_l}: \; \mathcal{M}_l \to |\mathcal{O}(l)|$, which is defined
for $l=1,2,3$.
\begin{thm}
 Assume $l \in \{1,2\}$. Then the morphism $p$ is flat. Further, the generic fibre of $p$ is irreducible. Hence $\mathcal{M}_l$ is irreducible for $l=1,2$.
\end{thm}
\begin{proof}
We will make use of the following result from \cite[prop. 6.1.5]{egaiv}: suppose $f: X \to Y$ is a morphism between locally Noetherian schemes
with $X$ Cohen--Macaulay and $Y$ regular, and further suppose that for each $x \in X$ the dimension of the scheme theoretic fibre $X_{f(x)}$
is constant of dimension $\dim_x X-\dim_{f(x)} Y$. Then $f$ is flat.

Since $\mathcal{M}_l$ and $|\mathcal{O}(l)|$ are smooth, it is enough to show that nonempty fibres of $p$ have constant dimension 
$$\dim(\mathcal{M}_l) - \dim (|\mathcal{O}(l)|)=\dim(|\mathcal{O}(4)|)-\dim (|\mathcal{O}(l)|)=34-\binom{l+3}{3}+1$$ (there are only finitely many rational curves
in any linear system on a K3 surface, so that $\dim(\mathcal{M}_l)=\dim(|\mathcal{O}(4)|)$).
Let $S \in |\mathcal{O}(l)|$ lie in the image of $p$. The fibre $(\mathcal{M}_l)_S$ over $S$ can be identified with the subspace of $|\mathcal{O}(4)|$
consisting of smooth quartics $X$ with $X \cap S$ an irreducible, rational, nodal curve. For this to be nonempty, we must have that $S$ is integral. 

Now let $W_S$ denote the space of irreducible, rational, nodal curves on $S$ in the linear system $|\mathcal{O}_S (4)|$. 
We have a morphism  $t: (\mathcal{M}_l)_S \to W_S$, $X \mapsto X \cap S$. Let $C \in W_S$ represent an irreducible, rational, nodal curve. Then we may view the fibre $t^{-1}(C)$
as an open subset of $\proj(Ker(H^0(\mathcal{O}(4))) \to H^0(\mathcal{O}_C (4)))$. Since $C$ is a complete intersection of a quartic and a hypersurface of degree $l$, the restriction morphism 
$H^0(\mathcal{O}(4)) \to H^0(\mathcal{O}_C (4))$ is surjective,
see for example \cite[p.\ 206, ex.\ 3.3]{liu}. Further, an easy calculation gives $\dim(Ker(H^0(\mathcal{O}(4))) \to H^0(\mathcal{O}_C (4)))=21$ for $l=1$ and  
$\dim(Ker(H^0(\mathcal{O}(4))) \to H^0(\mathcal{O}_C (4)))=11$ for $l=2$. So the fibres of $p$ are of constant dimension $\dim(W_S)+20=31=35-\binom{4}{3}$ for $l=1$ and
$\dim(W_S)+10=25=35-\binom{5}{3}$ for $l=2$; and hence we have shown flatness of $p$.

Now let $S \in |\mathcal{O}(l)|$ be generic; then $S$ corresponds to a smooth surface. In this case $W_S$ is smooth and irreducible, since for $l=1$, $S$ is a plane $\proj^2$, and for $l \in \{2,3\}$,
$S$ is a Del Pezzo surface (and not the one exception of Testa's paper). See \cite{testa} for the proof in the
Del Pezzo case.  The morphism  $t: (\mathcal{M}_l)_S \to W_S$, $X \mapsto X \cap S$ constructs the fibre of $p$ as an open set of the projective bundle over $W_S$ with fibre over $C$ given by
 $\proj(Ker(H^0(\mathcal{O}(4))) \to H^0(\mathcal{O}_C (4))).$ Hence the generic fibre of $p$ is irreducible as required.
\end{proof}

\begin{rem}
A smooth cubic hypersurface in $\proj^3$ is still a Del Pezzo surface, so the above 
arguments show that the space 
\begin{align*}
\mathcal{M}_3^0:=& \{ (X,S)  \; | \; X \in |\mathcal{O}(4)|  \; \text{smooth},\; S \in |\mathcal{O}(3)|  \; \text{smooth cubic surface} \\
    & \, \text{with \;} X \cap S \text{\; a rational, irreducible, nodal curve} \}
 \end{align*} 
is irreducible. The problem is that it is not clear that $\mathcal{M}_3^0 \seq \mathcal{M}_3$
is dense. The biggest difficulties come from nonnormal cubic surfaces, which can
have much larger Severi varieties than smooth cubic surfaces.
\end{rem}

\section{The moduli space of stable maps} 
In this section we will review the moduli space of stable maps. We will need this theory in order
to study deformations of rational curves.

Let $C$ be a scheme of dimension one, proper over $k=\C$. We define the \emph{arithmetic genus} to be
$p_a(C):=H^1(C,\mathcal{O}_C)$. We call $C$ a \emph{prestable curve} if it satisfies the
following conditions
\begin{enumerate}
\item $p_a(C)=0$.
\item $C$ is reduced and connected with at worst nodes.
\end{enumerate}
Note that if $C_k$ is any component of a prestable curve $C$ then $p_a(C_k)=0$, so that $C_k$ is a 
smooth rational curve.

Recall the following definition.
\begin{mydef}
Let $S$ be a Noetherian scheme over $k=\C$, $X$ a projective scheme over $S$, and let $H$ be an $S$ ample divisor on $X$. 
A stable map of degree $d$ over $U$ is a pair $( \sigma : C \to U, f: C \to X \times_S U)$ where $U$ is an $S$ scheme and
$f$ and $\sigma$ are morphisms subject to the conditions
\begin{enumerate}
\item $\sigma$ is flat and proper and $p_2 \circ f = \sigma$ where $p_2:X \times_S U \to U$ is the projection.
\item For each closed point $Spec(k) \to U$, $C_k$ is a prestable curve and $(f_{k*} C_k \cdot H_k)=d$, where $f_k: C_k \to X_k$
is the induced map.
\item Let $Spec(k) \to U$ be a closed point and let $Z_k$ be a component of $C_k$ such that the induced map $f_k: C_k \to X_k$ maps 
$Z_k$ to a single point. Then $Z_k$ must intersect at least three other components of $C_k$.
\end{enumerate} 
\end{mydef} 

One may construct a moduli space parametrising stable maps into a projective scheme. 
The following is shown in \cite[ch.10]{arakol}.
\begin{prop}
Let $S$ be a Noetherian scheme over $\C$, and let $X$ be a projective scheme over $S$.
Let $H$ be an $S$ ample line bundle on $X$. Then there is a scheme $T(X/S,d)$ which
is a coarse moduli space for the functor $F_d$. Further, the scheme $T(X/S,d)$ is projective
over $S$.
\end{prop}

We would like to develop some results about the deformations theory of maps. In the characteristic $p$ case, the required results appear in \cite{bogomolov}, but the proof essentially remains true in characteristic zero (and in fact is slightly simpler). We
repeat the results here for lack of a suitable reference. 

 Let $X$ be
a K3 surface and let $f: C \to X$ be a \emph{generically unramified} rational stable map, i.e. a morphism
such that $f$ is unramified when restricted to the generic points of the irreducible components of $C$. In characteristic zero,
this simply means that for each component $C_k$ of $C$, $f(C_k)$ is not a point.
We define $T^i_C=\mathrm{Ext}^i(\Omega^1_C,\mathcal{O}_C)$. Note that if $C$ is a \emph{smooth} rational curve then we
have $T^i_C=H^i(C,T_C)$ where $T_C$ is the tangent bundle, and we also have a short exact sequence
$0 \to T_C \to f^*(T_X) \to N_f \to 0$, where $N_f$ is the normal sheaf. In particular
$$ \chi(T_C) +\chi(N_f)=\chi(f^*T_X).$$ We will now
show that the above equation can be generalised to the case of a generically unramified rational stable map $f: C \to X$. For such a map
$f$, the complex $\R \mathcal{H}om_{\mathcal{O}_C}(\Omega_f^{\bullet}, \mathcal{O}_C)$ is supported in degree one, where
$\Omega_f^{\bullet}$ is the complex $df^t: f^* \Omega_X \to \Omega_C$ supported in degrees $-1$ and $0$, see 
\cite[Lemma 12]{bogomolov}. The degree one term of $\R \mathcal{H}om_{\mathcal{O}_C}(\Omega_f^{\bullet}, \mathcal{O}_C)$ is defined
to be the normal sheaf $N_f$. We warn the reader that this is
\emph{not} the same sheaf as the cokernel of the injective map $T_C \to f^*(T_X)$ which is sometimes defined to be
normal sheaf of $f$. In particular, $N_f$ for us is locally free when $f$ is unramified, but the latter
sheaf is never locally free for singular $C$. 

\begin{lem} \label{unram-lem}
Let $f: C \to X$ be a generically unramified rational stable map to a K3 surface. Then $ \dim T^0_C-\dim T^1_C +\chi(N_f)=\chi(f^*T_X)$.
\end{lem}
\begin{proof}
We use the calculations of \cite[Lemma 12]{bogomolov} for a generically
unramified morphism. In particular we have $\mathcal{E}xt^q_{\mathcal{O}_C}(f^* \Omega^1_X,\mathcal{O}_C)=0$
for $q >0$ and $\mathcal{E}xt^q(\Omega^1_C,\mathcal{O}_C)=0$ for $q>1$.
Now look at the local to global spectral sequence
$$H^p(C,\mathcal{E}xt^q(\Omega^1_C,\mathcal{O}_C)) \Rightarrow \mathrm{Ext}^{p+q}(\Omega^1_C,\mathcal{O}_C) .$$
This sequence consists of only four terms, namely $E^{0,0}=H^0(C,\mathcal{H}om(\Omega^1_C,\mathcal{O}_C))$,
$E^{0,1}=H^0(C,\mathcal{E}xt^1(\Omega^1_C,\mathcal{O}_C))$, $E^{1,0}=H^1(C,\mathcal{H}om(\Omega^1_C,\mathcal{O}_C))$
and $E^{1,1}=H^1(C,\mathcal{E}xt^1(\Omega^1_C,\mathcal{O}_C))$. In particular, all
differentials on the $E_2$ level are trivial. Hence we get $H^1(C,\mathcal{E}xt^1(\Omega^1_C,\mathcal{O}_C)) \simeq \mathrm{Ext}^{2}(\Omega^1_C,\mathcal{O}_C)=0$,
since $\Omega^1_C$ admits a locally free resolution $0 \to \mathcal{E}_1 \to \mathcal{E}_0 \to \Omega^1_C \to 0$,
(see \cite[Lemma 12]{bogomolov}). Next, we get $Hom(\Omega^1_C,\mathcal{O}_C) \simeq H^0(C,\mathcal{H}om(\Omega^1_C,\mathcal{O}_C))$
and $\mathrm{Ext}^1(\Omega^1_C,\mathcal{O}_C)) \simeq H^0(C,\mathcal{E}xt^1(\Omega^1_C,\mathcal{O}_C)) \oplus H^1(C,\mathcal{H}om(\Omega^1_C,\mathcal{O}_C))$
where the last isomorphism is non-canonical. 

Now look at the local to global spectral sequence
$$H^p(C,\mathcal{E}xt^q(f^*\Omega^1_X,\mathcal{O}_C)) \Rightarrow \mathrm{Ext}^{p+q}(f^*\Omega^1_X,\mathcal{O}_C) .$$
This sequence has only two terms, $H^0(C,\mathcal{H}om(f^* \Omega^1_X,\mathcal{O}_C))$
and $H^1(C,\mathcal{H}om(f^* \Omega^1_X,\mathcal{O}_C))$. Hence we see that
$\mathrm{Hom}(f^*\Omega^1_X,\mathcal{O}_C)) \simeq H^0(C,\mathcal{H}om(f^*\Omega^1_X,\mathcal{O}_C))$
and $\mathrm{Ext}^{1}(f^*\Omega^1_X,\mathcal{O}_C))\simeq H^1(C,\mathcal{H}om(f^*\Omega^1_X,\mathcal{O}_C))$.
Further, from \cite[Lemma 12]{bogomolov}, we have the following
two short exact sequences of sheaves on $C$:
\begin{equation} \label{bog-ses1}
 0 \to \mathcal{H}om(\Omega_C^1,\mathcal{O}_C) \to \mathcal{H}om(f^*\Omega^1_X, \mathcal{O}_C)
\to \mathcal{H}om(f^*\Omega^1_X,\mathcal{O}_C) / \mathcal{H}om(\Omega_C^1,\mathcal{O}_C) \to 0
\end{equation}
and 
\begin{equation} \label{bog-ses2}
0 \to \mathcal{H}om(f^*\Omega^1_X,\mathcal{O}_C) / \mathcal{H}om(\Omega_C^1,\mathcal{O}_C)
\to N_f \to \mathcal{E}xt^1(\Omega^1_C,\mathcal{O}_C) \to 0.
\end{equation} 
The result now follows from the additivity of the Euler characteristic $\chi$ and the above short exact sequences.
\end{proof}

We now wish to describe the deformation theory of a generically unramified stable map to a K3 surface.
As above let $f: C \to X$ be a morphism, and let $Def(f)$ be the functor of deformations
of $f$, see \cite{ran-maps}. Here we are allowing both the domain $C$ and the target $X$ to deform. 
Denote by $T^0_f$, $T^1_f$, and $T^2_f$ the complex
vector spaces parametrising infinitesimal automorphisms, deformations and obstructions
respectively. Let $Def(X)$ be the 20 dimensional,
universal deformation space of the K3 surface $X$.
 
\begin{prop} \label{basic-est}
Assume $f: C \to X$ is a generically unramified stable map to a smooth K3 surface
and let $N_f$ be the normal sheaf to $f$. Then
$T^0_f=0$. Further $\dim T^1_f-\dim T^2_f=\chi(N_f)+\dim(Def(X))$. 
\end{prop}
\begin{proof}
We follow the notation of \cite{ran-maps}.
Since $X$ is assumed smooth
we have that $f^*\Omega^1_X$ is locally free and so $L^q f^*\Omega^1_X=0$
for $q >0$ where $L^q f^*$ denotes the left-derived functors of $f^*$.
Hence the spectral sequence $\text{Ext}^p_C(L^q f^*\Omega^1_X,\mathcal{O}_C) \Rightarrow \text{Ext}_f^i(\Omega^1_X,\mathcal{O}_C)$
gives $\text{Ext}_f^0(\Omega^1_X,\mathcal{O}_C)=H^0(C,f^*T_X)$ and $\text{Ext}_f^1(\Omega^1_X,\mathcal{O}_C)=H^1(C,f^*T_X)$.
Write $T^i_C$ for $\text{Ext}^i(\Omega^1_C,\mathcal{O}_C)$ and $T^i_X$ for $H^i(X,T_X)$.
Then, by \cite[Sect.2]{ran-maps} we get a long exact sequence
\begin{equation} \label{def-ses}
0 \to T^0_f \to T^0_C \to H^0(f^*T_X) 
 \to T^1_f \to T^1_C \oplus T^1_X \to H^1(f^*T_X) \to T^2_f \to 0
\end{equation}
 since $T^0_X$ and $T^2_X$ vanish
for a K3 surface and $\text{Ext}^2(\Omega^1_C,\mathcal{O}_C)=0$ from the proof of Lemma \ref{unram-lem}.
The sequence \ref{bog-ses1} from the proof of the previous lemma gives us that the map
$$ T^0_C \to H^0(C,f^*T_X)$$
is injective and hence $T^0_f=0$, so that $Def(f)$
admits no infinitesimal automorphisms. Thus taking the Euler characteristic of 
\ref{def-ses} shows $$\dim T^0_C-\dim T^1_C+\dim T_f^1-\dim T_f^2+\dim T^1_X=\chi(f^*T_X).$$
Combining this with Lemma \ref{unram-lem} gives the result, noting that $\dim(Def(X))=\dim T^1_X$.
\end{proof}

We now let $\mathcal{X} \to Def(X)$ be the universal deformation of a K3 surface,
and let $Def(X,L)$ denote deformations of the \emph{polarised} K3 surface $(X,L)$.
Then $Def(X,L) \seq Def(X)$ has codimension one. Fix some Hodge class $B \in \Lambda$,
where $\Lambda$ is the K3 lattice. As in the proof of \cite[Thm.19]{bogomolov} or from \cite[Sect.8]{vist-ab}
we can construct a stack $T(\mathcal{X} / Def(X),B)$, proper over $Def(X)$, 
parametrising families of stable curves $f: C \to Y$ with $Y \in Def(X)$ and
such that $f_*(C)=B \in \Lambda$.

In the next lemma we assume that $f$ is unramified in the stronger sense (and not merely
generically unramified).
\begin{cor} \label{unram-surj}
Let $f:C \to X$ be an unramified, stable map with $f_*(C)=B \in \Lambda$. Then every irreducible component
$I$ of $T(\mathcal{X} / Def(X),B)$ containing $f$ surjects onto $Def(X,L)$.
\end{cor}
\begin{proof}
The obstructions and deformation spaces of the functor $Def(f)$ correspond to
those of $T(\mathcal{X} / Def(X),B)$. Indeed an element of $Def(f)(S)$ over a pointed
analytic space $(S,0)$ consists
of flat families $g:\tilde{C} \to S$ and $h:\tilde{X} \to S$ and an $S$-morphism
$\tilde{f}: \tilde{C} \to \tilde{X}$ with fibre over $0 \in S$ corresponding
to $f: C \to X$. Under these circumstances, there is a small open analytic space
$T \seq S$ such that $\tilde{f}|_T$ is a deformation of stable maps (i.e.\ the stability
condition is an open condition). 
Hence Proposition \ref{basic-est} shows that $I$ has dimension at least $\chi(N_f)+\dim(Def(X))$,
c.f. \cite[Sect.\ 11]{har-defo}. Since $f$ is unramified, $\chi(N_f)=-1$, see \cite[Lemma 14]{bogomolov}.
Thus $\dim(I) \geq \dim(Def(X))-1=\dim(Def(X,L))$.

Since $f$ is an unramified morphism, the fibre of $T(\mathcal{X} / Def(X),B) \to Def(X)$ over $X$ is
zero dimensional at $f$, by \cite[Lemma 14]{bogomolov} or \cite[Lemma 2.6]{liedtke}. 
Note that the condition that $f$ is unramified everywhere automatically gives us that $f|_{C_k}$ is birational onto its image for each irreducible component
$C_k$ of $C$ by the Hurwitz theorem. 
Since $\pi: T(\mathcal{X} / Def(X),B) \to Def(X)$
is proper, $\pi(I)$ is a closed subset of $Def(X)$ with dimension equal to $\dim(Def(X,L))$.
Now, if $t: C_1 \to X_1 \in T(\mathcal{X} / Def(X),B)$ then $B=t_*(C_1) \in Pic(X_1) \seq H^2(X_1, \C)$ is
a $(1,1)$ class. Thus $\pi(I) \seq Def(X,L)$. Since $Def(X,L)$ is irreducible, we see 
that every irreducible component
$I$ of $T(\mathcal{X} / Def(X),B)$ surjects onto $Def(X,L)$ as required.  
\end{proof}

\section{The space $\overline{M}_{0,k}^g$} \label{sect:main}
Our first task is to study the space $\overline{M}_{0,k}^g$ defined in the introduction to this chapter. We will show that we gain some flexibility by looking at this space, rather than $M_{0,k}^g$. 
The space $M_{0,k}^g$ is a (dense) open subset of $\overline{M}_{0,k}^g$, because there is an open subset of $|k\mathcal{L}|$ parametrising irreducible curves with 
at worst nodes. Thus the space $M_{0,k}^g$ is irreducible precisely when $\overline{M}_{0,k}^g$ is irreducible. So, to study questions about the irreducibility of $M_{0,k}^g$,
we lose nothing by studying $\overline{M}_{0,k}^g$ instead.

We start by studying the points in $\overline{M}_{0,k}^g$.
\begin{lem}
 Let $(X,L)$ be a primitively polarized K3 surface of genus $g$, and $C \seq X$ a curve in $|kL|$ such that $(C,X,L) \in \overline{M}_{0,k}^g$. 
 Then the curve $C$ has rational components.
\end{lem}
\begin{proof}
 This follows immediately from the fact that a limit of rational curves is still rational, see \cite[p.\ 103]{kollar}.
\end{proof}

\begin{mydef}
We will use the term rational curve to denote any one dimensional, proper scheme over $\C$ with rational components. We do not
assume either that $C$ is irreducible nor that $C$ is reduced.
\end{mydef}

Let $(X,L)$ be a primitively polarized K3 surface and $C \in |kL|$ an arbitrary rational curve. 
One may ask whether it is true that $(C,X,L)$ is represented by an element of $\overline{M}_{0,k}^g$.
We are not sure which conditions, if any, would be necessary for a rational curve $C$ to represent a point in $\overline{M}_{0,k}^g$, but we are able to give one simple 
sufficient condition which
is already quite useful. Denote by $T(X,kL)$ the moduli stack of genus zero stable maps $f: D \to X$ with $f_{*}(D) \in |kL|$. We say $C \in |kL|$
admits an unramified representative if there is an unramified map $f \in T(X,kL)$ with scheme theoretic image equal to $C$. 

\begin{prop}
 Let $(X,L)$ be a primitively polarized K3 surface of genus $g$ and $C \in |L|$ a rational curve admitting an unramified representative. 
 Then $(C,X,L) \in \overline{M}_{0,1}^g$.
\end{prop}
\begin{proof}
From Corollary \ref{unram-surj},
$f$ can be deformed to a stable map on a generic K3 surface. However, a generic K3 surface
admits only irreducible, nodal curves in the primitive class by \cite{chen}. Further, there
is a morphism from the semi-normalization of the relative moduli space $T(\mathcal{X} / W_g, \mathcal{L})$ to
$\overline{M}_{0,1}^g$ by \cite[I.6]{kollar}, where we are using the notation of Section \ref{sect:const}. Hence $(C,X,L) \in \overline{M}_{0,1}^g$.
\end{proof}

\section{Elliptic K3 surfaces}
We start with the following definition:
\begin{mydef}
 Let $X$ be a surface and $C$ be a smooth curve over $\C$. A Weierstrass fibration is a flat and proper map $\pi: X \to C$
such that 
\begin{enumerate}
 \item Each fibre is an integral curve with arithmetic genus one.
\item The general fibre is smooth.
\item There is a section $S$ of $\pi$ which avoids all nodes and cusps of the fibres.
\end{enumerate}
Furthermore, the Weierstrass fibration is said to be in minimal form in the case that $X$ has only rational double points as singularities.
\end{mydef}

An \emph{elliptic K3 surface} is a K3 surface $X$, possibly with rational double points, together with a minimal Weierstrass fibration $\pi: X \to \proj^1$ with section $S$. 
Let $E$ be the class of a generic
fibre and consider the linear system $|S+gE|$.
\begin{lem}
 Let $X$ be a smooth, elliptic K3 surface with fibre $E$ and section $S$. Then the line bundle $S+gE$ is nef and big for $g > 2$, 
and further in this case every divisor $D$ in the linear system $|S+gE|$ takes the form
$S + \sum_{i=1}^g F_i$ where $F_i$ are fibres of $\pi$ (we allow repetitions of the fibres).
\end{lem}
 \begin{proof}
The section $S$ is a smooth rational curve on a K3 surface so that $(S \cdot S)=-2$, and the curve $E$ is smooth and elliptic so that $(E \cdot E)=0$. 
We have $(S+gE \cdot S+gE)=(S \cdot S) +2g(S \cdot E)+g^2(E \cdot E)=-2+2g$. Hence $S+gE$ is nef and big for $g > 2$ and thus $H^1(X, \mathcal{O}(S+gE))=0$. 
The dimension of the linear system $|S+gE|$ is equal to $\frac{1}{2}(-2+2g)+1=g$ by Riemann--Roch. But this is the same as the dimension of the (closed)
subscheme of the projective space $|S+gE|$ consisting of divisors of the form $S + \sum_{i=1}^g F_i$ for $F_i$ fibres of $\pi$. Hence each element
of $|S+gE|$ takes this form.
 \end{proof}

Note from this proof that $(S+gE \cdot S+gE)=2g-2$. Thus if $S+gE$ is ample and primitive, we have $(X,S+gE) \in W_g$. The next lemma gives us
a simple condition under which this holds.
\begin{lem}
 Let $X$ be a smooth, elliptic K3 surface. Let $S$ be a section of $\pi$ and $E$ a generic fibre. Then the divisor $D=S+gE$ is ample 
and the associated line bundle $\mathcal{O}(S+gE)$ is primitive, for
all $g > 2$.
\end{lem}
\begin{proof}
We firstly establish primitivity. Suppose $M \in Pic(X)$ is such that $M^n=\mathcal{O}(S+gE)$ for some positive integer $n$. We
have $n(M \cdot E)=(M^n \cdot E)= (S+gE | E)=1$, which forces $n=1$ since $(M \cdot E)$ is an integer.

We next establish ampleness. The Nakai--Moishezon criterion says that a divisor $D$ on a surface $X$ is ample if and only if
$D^2 > 0$ and $(D \cdot C) >0$ for each irreducible curve $C$ in $X$, see \cite[p.\ 365]{har}.
Now we have $(S+gE \cdot S+gE)=2g-2>0$ for $g \geq 2$, 
and the first condition is satisfied. Let $C$ be an irreducible curve. Then $\pi(C)$ is irreducible, and hence either $\pi(C)=\proj^1$ or $\pi(C)=\{p\}$ for some point $p \in \proj^1$. In the 
first case $C$ must meet the generic fibre $E$ properly, and hence $(C \cdot E) >0$.  If $C \neq S$, then $(C \cdot S+gE) \geq (C \cdot gE) > 0$.
If $C=S$ we have $(S \cdot S+gE)=-2+g>0$, for $g>2$. 
Next assume $\pi(C)=\{p\}$, i.e.
$C$ is a fibre. We then have $C \sim E$ (since the fibres are assumed integral) and so $(C \cdot S+gE)=1>0$.
Hence $D=S+gE$ is ample.
\end{proof}
  
Now suppose $X$ is a smooth elliptic K3 surface such that all singular fibres are nodal. Note that, since each fibre is
integral of arithmetic genus one, the singular fibres have precisely one node and are rational. Let 
the singular fibres be $N_1, \ldots, N_t$. The Euler characteristic of each fibre is $\chi(N_i)=1$ and further $24=\chi(X)=\sum_{i=1}^t \chi(N_i)$, from
which we see that there 
are precisely $24$ singular fibres $N_1, \ldots, N_{24}$. Every
rational curve in $|S+gE|$ has the form $S+ \sum_{i=1}^{24} m_i N_i$ with $\sum_{i=1}^{24} m_i=g$. 
We will show that each of these curves has an unramified representative.
\begin{prop}
 Let $X$ be a smooth elliptic K3 surface, and suppose that each singular fibre $N_i$ is nodal. Let $C \in |S+gE|$ be rational.
Then there exists a unramified stable map $f: D \to X \in T(X, S+gE)$ representing $C$.
\end{prop}
\begin{proof}
 Suppose $C$ has the form $S+ \sum_{i=1}^{24} m_i N_i$. We construct $f$ as follows. 
 Assume $j$ is such that $m_j \neq 0$ and let $f_{j,1}: D_{j,1} \to N_j$ be the normalization morphism. 
Now suppose $p_{j_1}$ is a point on $D_{j,1}$ such that $f_{j,1}$ maps $p_{j_1}$ to the node. 
Then attach a smooth rational curve $D_{j,2}$ to $D_{j,1}$ at the point $p_{j_1}$, 
and define a map $f_{j,2}: D_{j,1} \cup D_{j,2} \to 2N_j$ 
by demanding that $f_{j,2}$ is locally biholomorphic, i.e.\ in a small analytic open set around $p_j$, 
$f_{j,2}$ maps $D_{j,1}$ and $D_{j,2}$ to two separate branches of the node at $N_j$. Now let $p_{j_2} \neq p_{j_1}$ be 
the other point on $D_{j,2}$ which is sent to the node, attach a smooth 
rational curve $D_{j,3}$ to $p_{j_2}$ and define a locally biholomorphic map $f_{j,3}: D_{j_1} \cup D_{j,2} 
\cup D_{j,3}  \to 3N_j$.

 Continuing in this way, we
can construct a map $f_{j,m_j}: T_j:=D_{j,1} \cup D_{j,2} \cup \ldots \cup D_{j,m_j} \to m_j N_j$. 
Let $T_0$ be smooth and rational, and $g_0: T_0 \to S$ be an isomorphism. Define a connected curve
$D:=T_0 \cup_j T_j$ where we glue $T_j$ to $T_0$ at the point of intersection between $S$ and $N_j$. 
There is then a unique map
$f: D:=T_0 \cup_j T_j \to C$, where the index ranges over those $j$ with $m_j \neq 0$, such that $f|_{T_j}=f_{j,m_j}$
and $f|_{T_0}=g_0$. 
Note that on each component $I$ of $D$, $f|_{I}$ is 
unramified since the normalization of an integral nodal curve is unramified and birational. 
Furthermore, it is clear that $f$ is biholomorphic to its image in an analytic neighbourhood around all nodes of $T$. Hence $f$ is unramified.
See the figure below.
\end{proof}
\begin{figure}[h]
 \centering
 \includegraphics{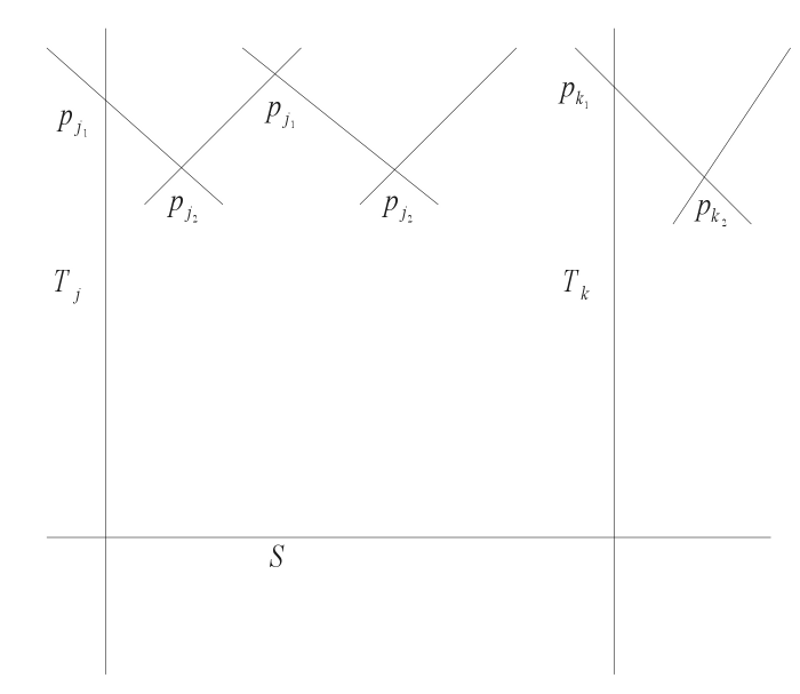}
\caption{The unramified representative of $C \in |S+gE|$ rational}
\label{unfam}
\end{figure}
\begin{cor} \label{nodal-cor}
 Let $X$ be as above. Let $C \in |S+gE|$ be rational. Then $(C,X,S+gE) \in \overline{M}_{0,1}^g$.
\end{cor}

\section{Explicit constructions of elliptic K3 surfaces}
In this section we will use the theory of Weierstrass fibrations to give some explicit constructions of elliptic K3 surfaces. 
The starting point of the theory of Weierstrass fibrations is the Weierstrass equation. Consider the plane curve defined
by
 $y^2 z=x^3+ax z^2+bz^3$
for $a, b \in \C$. This is called the \emph{Weierstrass equation}, and it defines a cubic in $\proj^2$.
This cubic is singular for values of $a,b$ if and only if the discriminant $\Delta=-16(4 a^3+27 b^2)$ vanishes. When the 
discriminant is nonzero, it defines a smooth elliptic curve.  

We now wish to define the global Weierstrass equation of a Weierstrass fibration over $\proj^1$. Let $L:=\mathcal{O}_{\proj^1}(N)$ for $N>0$,
and let $(A,B) \in H^0(L^4) \times H^0(L^6)$ be pairs of global sections, such that $\Delta(A,B):=-16(4 A^3+27 B^2) \in H^0(\mathcal{O}_{\proj^1}(12N))$ is not
the zero section. Now consider the closed subscheme $X$ of $\proj(L^{\otimes 2}\oplus L^{\otimes 3} \oplus \mathcal{O}_{\proj^1})$
defined by
$ y^2 z=x^3+Ax z^2+Bz^3$
where $(x,y,z)$ is a global coordinate system of $\proj$ relative to $(L^{\otimes 2} , L^{\otimes 3} , \mathcal{O}_{\proj^1})$. 
Assume further that for each $p \in \proj^1$, either $\mu_p(A) \leq 3$ or $\mu_p(B) \leq 5$ (or that both conditions hold).
Then $X \to \proj^1$
is a Weierstrass fibration in minimal form, and conversely any such Weierstrass fibration can be described in the above way. Further,
$X$ is a (possibly singular) K3 surface precisely when $N=2$. The singular fibres of $X$
occur precisely over those $q \in \proj^1$ with $\Delta(A,B)(q)=0$. For proofs of these facts, see \cite[II.5, III.3, III.4]{miranda}. 

The above method produces explicit equations for a Weierstrass fibration. In the following
we assume $N=2$ since we are only interested in the K3 case. We start
by describing the $\proj^2$ bundle $\proj(\mathcal{O}_{\proj^1}(4) \oplus \mathcal{O}_{\proj^1}(6) \oplus \mathcal{O}_{\proj^1})$
explicitly. Let $U_0 \simeq \mathbb{A}^1$ and $U_1 \simeq \mathbb{A}^1$ be two copies
of the affine line. Then $\proj^1$ can be obtained from $U_0 \sqcup U_1$ be identifying
$u_0 \in U_0$ with $u_1 \in U_1$ if and only if $u_0 \neq 0$ and $u_1= \frac{1}{u_0}$. Now set
$W_0 = U_0 \times \proj^2$ and $W_1 = U_1 \times \proj^2$. Then $\proj(\mathcal{O}_{\proj^1}(4) \oplus \mathcal{O}_{\proj^1}(6) \oplus \mathcal{O}_{\proj^1})$
is obtained from $W_0 \sqcup W_1$ by identifying $(u_0, [x_0: y_0: z_0])$ with 
$(u_1, [x_1: y_1: z_1])$ if and only if $u_0 \neq 0$ and 
$$ u_1= \frac{1}{u_0}, \; x_1=\frac{x_0}{u_0^4}, \; y_1= \frac{y_0}{u_0^6}, \; z_1=z_0.$$ The 
global sections $A \in H^0(\mathcal{O}_{\proj^1}(8))$ and $B\in H^0(\mathcal{O}_{\proj^1}(12))$ can be considered as pairs of polynomial functions $\{ A_0, A_1 \}$ and
$\{B_0, B_1\}$:
\begin{align*}
A_0: U_0 \to \mathbb{A}^1, \; \; A_1: U_1 \to \mathbb{A}^1 \\
B_0: U_0 \to \mathbb{A}^1, \; \; B_1: U_1 \to \mathbb{A}^1
\end{align*}
satisfying $A_1(u_1)=u_1^8 A_0(\frac{1}{u_1})$ and $B_1(u_1)=u_1^{12} B_0(\frac{1}{u_1})$.
The hypersurface $X$ is then given by the equations
$$
y_0^2 z_0 = x_0^3+A_0(u_0) x_0 z_0^2+ B_0(u_0)z_0^3
$$
on $W_0$ and
$$
y_1^2 z_1 = x_1^3+A_1(u_1) x_1 z_1^2+ B_1(u_1)z_1^3
$$
on $W_1$. The above equations make sense, since if $(u_0, [x_0, y_0, z_0])$ satisfy
the first equation, then $(\frac{1}{u_0}, [\frac{x_0}{u_0^4}, \frac{y_0}{u_0^6}, z_0])$
satisfies the second equation.

In particular, we can use the above explicit equations to show that certain Weierstrass fibrations
are smooth.
\begin{prop} \label{cusp-eg}
Let $a_1, \ldots, a_{12}$ be pairwise distinct complex numbers. 
Then the Weierstrass fibration defined by $A=0 \in H^0(\mathcal{O}_{\proj^1}(8))$ and 
$B= \prod_{i=1}^{12}(u_0-a_iu_1)\in H^0(\mathcal{O}_{\proj^1}(12))$ defines a smooth,
elliptic K3 surface. We denote
this K3 surface by $X_{(a_i),0}$. The
fibration $X_{(a_i),0} \to \proj^1$ has $12$ singular fibres, over the points $a_i \in U_0 \seq \proj^1$.  
\end{prop}
\begin{proof}
The discriminant in this case is $\Delta(A,B)=-432(B^2) \in H^0(\mathcal{O}_{\proj^1}(24))$ which vanishes
to order two at the $12$ distinct points $a_i \in U_0$, and hence $X_{(a_i),0}$ has $12$ singular fibres over the $a_i$.
The surface $X_{(a_i),0}$ is given by the equations
$$
f_0:=y_0^2 z_0 - x_0^3-\prod_{i=1}^{12}(u_0-a_i)z_0^3
$$
on $W_0$ and
$$
f_1:=y_1^2 z_1 - x_1^3- \prod_{i=1}^{12}(1-a_iu_1)z_1^3
$$
on $W_1$. 
We have
\begin{align*}
\frac{\partial f_0}{\partial x_0} &= -3x_0^2 \\
\frac{\partial f_0}{\partial y_0} &=2y_0 z_0 \\
\frac{\partial f_0}{\partial z_0} &= y_0^2-3 \prod_{i=1}^{12}(u_0-a_i)z_0^2 \\
\frac{\partial f_0}{\partial u_0} &= \sum_{l=1}^12 \prod_{i=1, i \neq l}^{12}(u_0-a_i) z_0^3
\end{align*}
and then the condition that the $a_i$ are pairwise distinct gives that $\frac{\partial f_0}{\partial x_0}$, 
$\frac{\partial f_0}{\partial y_0}$, $\frac{\partial f_0}{\partial z_0}$,
$\frac{\partial f_0}{\partial u_0}$ and $f_0$ have no common zeroes. Likewise, we have
\begin{align*}
\frac{\partial f_1}{\partial x_1} &= -3x_1^2 \\
\frac{\partial f_1}{\partial y_1} &=2y_1 z_1 \\
\frac{\partial f_1}{\partial z_1} &= y_1^2-3 \prod_{i=1}^{12}(1-a_iu_1)z_1^2 \\
\frac{\partial f_1}{\partial u_1} &= \sum_{l=1}^12 a_l\prod_{i=1, i \neq l}^{12}(1-a_i u_1) z_0^3
\end{align*}
and again the condition that the $a_i$ are pairwise distinct gives that $\frac{\partial f_1}{\partial x_1}$, 
$\frac{\partial f_1}{\partial y_1}$, $\frac{\partial f_1}{\partial z_1}$,
$\frac{\partial f_1}{\partial u_1}$ and $f_1$ have no common zeroes. Hence $X_{(a_i),0}$ is a smooth
K3 surface.
\end{proof}

Let $X \to \proj^1$ be a Weierstrass fibration defining a possibly singular K3 surface, and let $\tilde{X} \to \proj^1$ be a minimal desingularization of $X$. 
Then $\tilde{X}$ is a smooth, minimal elliptic surface
with section. 
The possible fibres $F$ of $\tilde{X}$ have been completely classified, see \cite{schuett}. The type of the fibre $F$ determines the singularities
of $X$. This information is tabulated below:

\begin{center}
\begin{tabular}{|l|p{5cm}|l|l|}
\hline
Kodaira type of $F$ & Description of $F$ & $\chi(F)$ & Singularity Type of $X$ \\ \hline
$I_0$ & smooth elliptic curve & $0$ & none \\ \hline
$I_1$ & irreducible, nodal rational curve & $1$ & none \\ \hline
$I_N$, $N \geq 2$ & cycle of $N$ smooth rational curves, meeting transversally & $N$ & $A_{N-1}$ \\ \hline
$I_N^*$, $N \geq 0$ & $N+5$ smooth rational curves meeting with dual graph $D_{N+4}$ & $N+6$ & $A_{N-1}$ \\ \hline
$II$ & a cuspidal rational curve & $2$ & none \\ \hline
$III$ & two smooth rational curves meeting at a point of order two & $3$ & $A_1$ \\ \hline
$IV$ & three smooth rational curves meeting at a point & $4$ & $A_2$ \\ \hline
$IV^*$ & seven smooth rational curves meeting with dual graph $\tilde{E}_6$ & $8$ & $E_6$ \\ \hline
$III^*$ & eight smooth rational curves meeting with dual graph $\tilde{E}_7$ & $9$ & $E_7$ \\ \hline
$II^*$ & nine smooth rational curves meeting with dual graph $\tilde{E}_8$ & $10$ & $E_8$ \\ \hline
 \end{tabular}
\end{center}
\clearpage

\begin{prop}
 Let $(A,B) \in H^0(\proj^1,\mathcal{O}(8)) \times H^0(\proj^1,\mathcal{O}(12))$ be data defining a Weierstrass fibration
in minimal form. Assume that the discriminant $\Delta=-16(4 A^3+27 B^2) \in H^0(\mathcal{O}_{\proj^1}(24))$ vanishes in $24$ distinct points
 of $\proj^1$. Then the Weierstrass fibration $X$ is a smooth K3 surface.  
\end{prop}
\begin{proof}
Let $\tilde{X}$ be the minimal desingularization of $X$, this is then a smooth K3 surface.
From the fact that the discriminant vanishes in $24$ points, $\tilde{X}$ must have $24$ singular fibres, $F_1, \ldots, F_{24}$. Further,
$24=\chi(\tilde{X})=\sum_{i=1}^{24} \chi(F_i)$, since $\tilde{X}$ is a smooth K3 surface. Hence each $F_i$ must have $\chi(F_i)=1$ and
by the above table must be an irreducible, nodal rational curve. Again from the above table, $X$ has no singularities, and thus is already smooth, so that
$\tilde{X} \simeq X$.
\end{proof}

\begin{eg} \label{nodal-eg}
 Let $(a_i)=(a_1, \ldots, a_{12}) \in \C^{12}$ be an ordered tuple of $12$ pairwise distinct points of $\C$, and let $K=r e^{i \theta}$ be a nonzero complex number,
with $r \in \R$, $0 \leq \theta < 2\pi$. Define $\delta:= -\sqrt[3]{\frac{27}{4}} \in \R$ and $\sqrt[3]{K^2}=\sqrt[3]{r} e^{\frac{2 \theta i}{3}}$ with
$\sqrt[3]{r} \in \R$. Let $x_0, x_1$ be homogeneous coordinates for $\proj^1$ and let 
\begin{align*}
\alpha &= \prod_{i=1}^{12}(x_0- a_i x_1) \in H^0(\mathcal{O}_{\proj^1}(12)) \\
A &= \sqrt[3]{K^2} \delta x_1^{8} \in H^0(\mathcal{O}_{\proj^1}(8)) \\
B &= \alpha + K x_1^{12} \in H^0(\mathcal{O}_{\proj^1}(12)).
\end{align*}
Then $A$ vanishes to order eight at $[1:0]$ whereas $B$ does not vanish at $[1:0]$, so the condition that either $\mu_p(A) \leq 3$ or $\mu_p(B) \leq 5$ is met.
Further $\Delta(A,B)=-432 (-K^2 x_1^{24} + (\alpha+K x_1^{12})^2) \in H^0(\mathcal{O}_{\proj^1}(24))$ which does not vanish identically. Thus the Weierstrass data $(A,B)$ defines
a minimal Weierstrass fibration $X_{(a_i), K} \to \proj^1$, such that $X_{(a_i), K}$ is a K3 surface with at most rational double point singularities. 
If we further assume that the form 
$\alpha + 2K x_1^{12}$
vanishes in $12$ distinct points of $p_1, \ldots , p_{12} \in \proj^1$, then $\Delta(A,B)$ vanishes at the $24$ distinct points 
$a_1, \ldots, a_{12}, p_1, \ldots, p_{12}$ (the $p_i$ are
distinct from $a_j$ by the assumption $K \neq 0$). Hence by the above proposition we then have that $X_{(a_i), K}$ is a \emph{smooth} K3 surface.
\end{eg}

\section{Connectedness of the space $\overline{M}^g_{0,1}$}
In this section we prove that the space $\overline{M}^g_{0,1}$ is connected.
We firstly simplify our notation by fixing $g$ and dropping the reference to the line bundle $T+gE$ on
an elliptic K3 surface.

\begin{mydef}
 Let $X \to \proj^1$ be a smooth elliptic K3 surface and $C$ a curve with rational components in $|S+gE|$,
 where $S$ denotes the section and $E$ the generic fibre. Then we define
$[C , X]:= (C, X, S+gE) \in |\mathcal{L}|$, where $|\mathcal{L}| \to W_g$ is defined in Section \ref{sect:const}.
\end{mydef}

Let $(a_i)=(a_1, \ldots, a_{12}) \in \C^{12}$ be an ordered tuple of $12$ pairwise distinct points of $\C$, 
define $p(t):= \prod_{i=1}^{12}(t-a_i)$ and assume $K$ is such that $p(t)=-2K$ has $12$ distinct solutions.
Consider the smooth, elliptic K3 surface $X_{(a_i),K}$ from Example \ref{nodal-eg}. We then have $[C, X_{(a_i),K}] \in \overline{M}^g_{0,1}$
by Corollary \ref{nodal-cor}. Let $\pi: X_{(a_i),K} \to \proj^1$ be the Weierstrass fibration and denote by $N_{a_i}$ the fibre of $\pi$
over $a_i$. Then $N_{a_i}$ is a nodal rational curve.

\begin{lem} \label{perm}
 Let $\sigma \in \mathfrak{S}_{12}$ be a permutation and let $(m_1, \ldots, m_{12})$ be any collection
of $12$ nonnegative integers such that $\sum {m_i} = g$. Then there is a path in $\overline{M}^g_{0,1}$ from 
$[S+\sum m_i N_{a_i}, X_{(a_i),K}]$ to $[S+ \sum m_i N_{\sigma(a_i)}, X_{(a_i),K}]$.
\end{lem}
\begin{proof}
First note that from the construction of $X_{(a_i),K}$ we clearly have $X_{(a_i),K}=X_{(\sigma(a_i)),K}$.
Now let $Y \seq \C^{12}$ be the open subvariety consisting of pairwise distinct tuples $(b_i)$, this is
obtained by removing all diagonals. Since $Y$ is irreducible, hence connected, there is a path $\gamma (t) \in Y$
with $\gamma(0)=(a_i)$ and $\gamma(1)=(\sigma(a_i))$. For any nonzero $K \in \C$, and for fixed integers $(m_1, \ldots, m_{12})$
adding to $g$, we can define a morphism 
\begin{align*} 
H_{K, (m_i)} \; : \; Y &\to \overline{M}^g_{0,1} \\
        (b_i) &\mapsto [S+\sum m_i N_{b_i}, X_{(b_i),K}].
\end{align*}
The composition $H_{K, (m_i)} \circ \gamma (t)$ is then our required path.
\end{proof}

We set $\alpha_n:= e^{\frac{\pi in}{6}}$ for $n=1,2,\ldots,12$. These are the
$12$ distinct roots of $t^{12}-1$. Note that $t^{12}-1=-2K$ has $12$ distinct 
roots so long as $K \neq \frac{1}{2}$, and so $X_{(\alpha_i),K}$ is a smooth K3 surface
with $24$ nodal fibres so long as $K \neq 0, \frac{1}{2}$. We now investigate the 
smooth elliptic K3 surface $X_{(\alpha_i),0}$ from Proposition \ref{cusp-eg}. The associated Weierstrass fibration has $12$
irreducible singular fibres (which are all cuspidal, rational curves) over $\alpha_i$. In keeping with the
above notation, denote the singular fibre over $\alpha_i$ by $N_{\alpha_i}$. 
\begin{lem} \label{perm-cusp}
 Let $\sigma \in \mathfrak{S}_{12}$ be a permutation and let $(m_1, \ldots, m_{12})$ be any collection
of $12$ nonnegative integers such that $\sum {m_i} = g$. Then $[S+\sum m_i N_{\alpha_i}, X_{(\alpha_i),0}] \in \overline{M}^g_{0,1}$,
and further there is a path in $\overline{M}^g_{0,1}$ from 
$[S+\sum m_i N_{\alpha_i}, X_{(\alpha_i),0}]$ to $[S+ \sum m_i N_{\sigma(\alpha_i)}, X_{(\alpha_i),0}]$.
\end{lem}
\begin{proof}
We can write $[S+\sum m_i N_{\alpha_i}, X_{(\alpha_i),0}] \in |\mathcal{L}|$ as a limit $[S+\sum m_i N_{\alpha_i}, X_{(\alpha_i),t}] \in \overline{M}^g_{0,1}$
as $t \mapsto 0$ (for $t \leq \frac{1}{2}$). Since $\overline{M}^g_{0,1}$ is closed in $|\mathcal{L}|$ it follows that 
$[S+\sum m_i N_{\alpha_i}, X_{(\alpha_i),0}] \in \overline{M}^g_{0,1}$. Further, this shows that there is a path in $\overline{M}^g_{0,1}$ between
$[S+\sum m_i N_{\alpha_i}, X_{(\alpha_i),\frac{1}{4}}]$ and $[S+\sum m_i N_{\alpha_i}, X_{(\alpha_i),0}]$ and likewise a path between 
$[S+\sum m_i N_{\sigma(\alpha_i)}, X_{(\alpha_i),\frac{1}{4}}]$ and $[S+\sum m_i N_{\sigma(\alpha_i)}, X_{(\alpha_i),0}]$. But from Lemma \ref{perm}
there is a path from $[S+\sum m_i N_{\alpha_i}, X_{(\alpha_i),\frac{1}{4}}]$ to $[S+\sum m_i N_{\sigma(\alpha_i)}, X_{(\alpha_i),\frac{1}{4}}]$,
so putting these three paths together we get the required path from 
$[S+\sum m_i N_{\alpha_i}, X_{(\alpha_i),0}]$ to $[S+ \sum m_i N_{\sigma(\alpha_i)}, X_{(\alpha_i),0}]$.
\end{proof}

The next step is to connect any rational curve $C$ in $X_{(\alpha_i),0}$ to the non-reduced rational curve $S+g N_{\alpha_1}$. 
\begin{lem} \label{steptwo}
 Notation as above. Let $[C,X_{(\alpha_i),0}] \in \overline{M}^g_{0,1}$. There is a path in $\overline{M}^g_{0,1}$ from 
 $[C,X_{(\alpha_i),0}]$ to $[T+gN_{\alpha_1},X_{(\alpha_i),0}]$.
\end{lem}
\begin{proof}
Let $\alpha_n:=e^{\frac{\pi i n}{6}}$, $n=1, 2, \ldots, 12$ be the $12$-th roots of unity.
We have that $C=S +\sum_{i=1}^{12} m_i N_{\alpha_i}$ with $\sum m_i=g$. Due to Lemma \ref{perm-cusp}, we may assume 
$m_1 \geq m_2 \geq \ldots \geq m_{12}$. If $m_2=0$ then we are done, so assume this is not the case. For $K$ close to zero, let $\beta(K)$ be a continuous
choice of solutions to $t^{12}=1-{2K}$ such that $\beta(K) \to 1$ as $K \to 0$.
Let $C_{K}$ be the rational curve $$S+m_1N_{\alpha_1}+N_{\beta(K) \alpha_2}+(m_2-1)N_{\alpha_2}+\sum_{i=3}^{12} m_i N_{\alpha_i} \; \in X_{(\alpha_i),K}$$ 
where we have extended out notation so that $N_{\gamma}$ denotes the singular fibre over $\gamma$ for those $\gamma \in \C$ with 
either $\gamma^{12}=1$ or $\gamma^{12}=1-2K$. Then 
 $$[C_{K},X_{(\alpha_i),K}] \longrightarrow [C,X_{(\alpha_i),0}] \in \overline{M}^g_{0,1} \; \; \text{as $K \to 0$}.$$ 
Now let $D_{K}$ denote the rational curve 
$$S+m_1N_{\alpha_1}+N_{\beta(K) \alpha_1}+(m_2-1)N_{\alpha_2}+\sum_{i=3}^{12} m_i N_{\alpha_i} \; \in X_{(\alpha_i),K}$$ and let $C'$ denote the curve 
$C+N_{\alpha_1}-N_{\alpha_2}$ in $X_{(\alpha_i),0}$.
We have $$[D_{K}, X_{(\alpha_i),K}] \longrightarrow [C',X_{(\alpha_i),0}] \in D \; \; \text{as $K \to 0$}.$$
It is enough to show that there is a path
from $[C_{K}, X_{(\alpha_i),K}]$ to $[D_{K}, X_{(\alpha_i),K}]$, since we would then have a path from $[C,X_{(\alpha_i),0}]$ to $[C',X_{(\alpha_i),0}]$, 
and repeating the process leads to the
desired path.

\begin{figure}[h]
 \centering
 \includegraphics{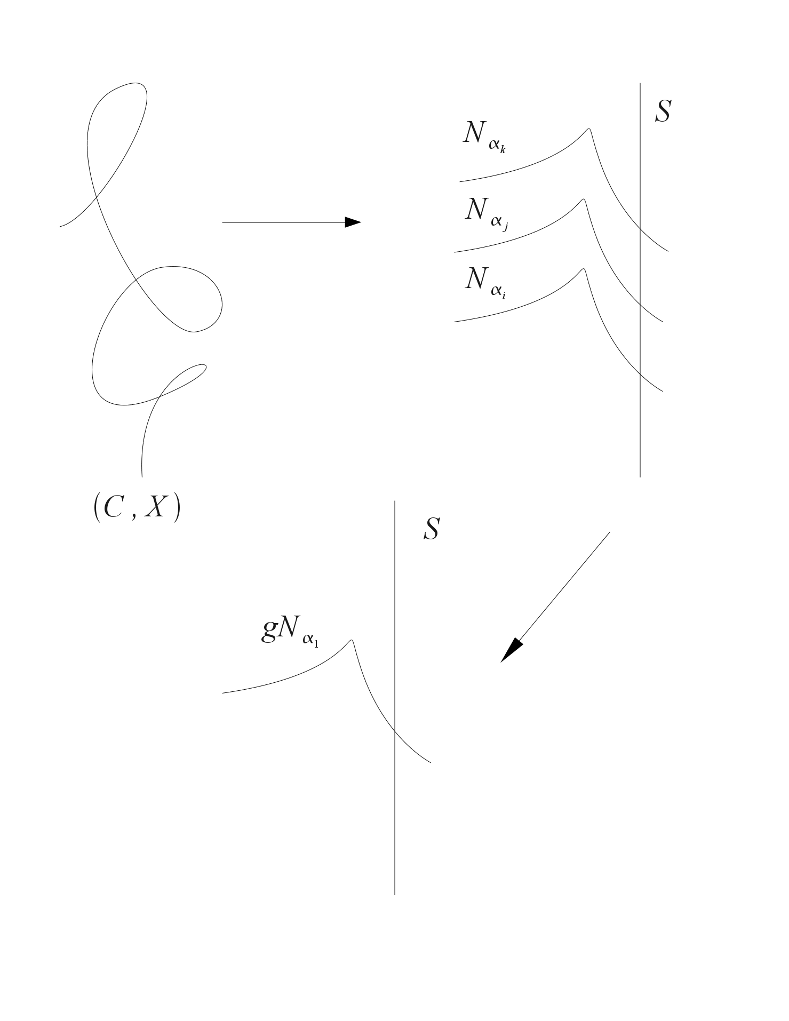}
\caption{Deforming a curve $C$ on a elliptic K3 surface $X$ of genus $g=3$ to $[S+gN_{\alpha_1},X_{(\alpha_i),0}]$}
\label{defo}
\end{figure}

We now finish the proof by constructing a path from $[C_{K}, X_{(\alpha_i),K}]$ to $[D_{K}, X_{(\alpha_i),K}]$ for small but fixed $K$.  
Set $\psi(s)=e^{\frac{\pi i s}{6}}$
and consider $X_{(\alpha_i),\frac{1}{2}(1-(\psi(s) \beta(K))^{12})}$, which for $s=0$ and $s=1$ is just $X_{(\alpha_i),K}$. Setting
$$F(s)=S+m_1N_{\alpha_1}+N_{\psi(s) \beta(K) \alpha_1}+(m_2-1)N_{\alpha_2}+\sum_{i=3}^{12} m_i N_{\alpha_i} \subseteq 
X_{(\alpha_i),\frac{1}{2}(1-(\psi(t) \beta(K))^{12})}$$ we have
that $\phi(s):=[F(s), X_{(\alpha_i),\frac{1}{2}(1-(\psi(s) \beta(K))^{12})}] \in \overline{M}^g_{0,1}$, and $F(0)=D_K$, $F(1)=C_K$. Hence
$[0,1] \to \overline{M}^g_{0,1}$, $s \mapsto \phi(s)$ gives the required path from $[C_{K}, X_{(\alpha_i),K}]$ to $[D_{K}, X_{(\alpha_i),K}]$.
\end{proof}

We now come to our main result.
\begin{thm}
The space $\overline{M}^g_{0,1}$ is connected.
\end{thm}
\begin{proof}
Let $W_g$ be the open subscheme of the Hilbert scheme parametrising genus $g$ polarized
K3 surfaces $(X,L)$. Then we have a proper projection map $p: \overline{M}^g_{0,1} \to W_g$, which is
smooth when restricted to $M^g_{0,1}$. Now let $I$ 
be an irreducible component of $\overline{M}^g_{0,1}$. Then $I=\overline{I_0}$ is the closure of
an irreducible component $I_0$ of $M^g_{0,1}$ in $|\mathcal{L}|$, and since $p|_{M^g_{0,1}}$
is flat and $W_g$ is irreducible, $\overline{p(I_0)}=W_g$. Thus $p(I)=W_g$, since $p$ is proper.

In particular, each connected component $B$ of $\overline{M}^g_{0,1}$ must contain at least one point
of the form $[C,X_{(\alpha_i),0}]$ for $C$ a curve with rational components in $|S+gE|$. But then by the above lemma 
$B$ contains \emph{all} such points. Since this applies to each connected component, and the connected components
are pairwise disjoint, there can only be one connected component.  
\end{proof}

\appendix
\chapter{The non-rational case} \label{sect:hilb}
In this section we use the Hilbert scheme of points to give a proof of the irreducibility of $V_{h}(X,kL)$ for curves with high geometric genus $h$. Let $S$ be a smooth
projective surface. We denote by $S^{[n]}$ the Hilbert scheme of zero dimensional closed subschemes of $S$ of length $n$. This 
is a compactification of the space of all unordered $n$-tuples of distinct points in $S$. Then $S^{[n]}$ is a smooth, irreducible, projective 
variety of dimension $2n$. For an introduction to the Hilbert scheme of points, see \cite[Ch.\ 1.1]{lehn}.

Let $S^{[3m]}$ denote the Hilbert scheme of points of length $3m$. Let $I_m \seq S^{[3m]}$ denote the closure
of the set with elements of the form $\coprod_{i=1}^m Spec (\mathcal{O}_{x_i}/m^2_{S,x_i})$ for $x_1, \ldots, x_m$ distinct points in $S$. 
Then $I_m$ is birational to the space $S^{[m]}$, by mapping a subscheme $Z$ to its reduction. Hence $I_m$ is irreducible. This space is important
to us for the following reason. Let $L \in Pic(S)$ be a line bundle, and $s \in H^0(S,L)$ a global section with $D=Z(s)$ the divisor of zeroes. 
Then $D$ has at least $m$ distinct singularities exactly when there exists some
 $[Z] \in I_m$ of the form $\coprod_{i=1}^m Spec (\mathcal{O}_{x_i}/m^2_{S,x_i})$ with $x_1, \ldots, x_m$ distinct and such that
 $s|_{Z}=0$. These singularities are generically nodes, but they
 may be worse.
     
We now recall the notion of a $k$ very ample vector bundle. A vector bundle $L \in Pic(S)$ is
said to be $k$ very ample if for each zero-dimensional subscheme $Z \seq S$ of length $k+1$, the restriction morphism
\[
 H^0(S,L) \to H^0(S, L \otimes \mathcal{O}_Z)
\]
is surjective. One can show that if $L$ is $k$ very ample for $k \geq 1$, then $L^{\otimes n}$ is $kn$ very ample. For more
on this topic, see \cite[Ch.\ 5]{goettsche}.

Let $L \in Pic(S)$ be a $3k-1$ very ample line bundle. Consider the closed subvariety $\mathcal{N}_m \seq I_m \times \proj(H^0(L))$ which parametrises pairs $(Z,s)$
with $s|_{Z}=0$. This can be constructed as a projective bundle over $I_m$ with fibre over $[Z]$ isomorphic to $\proj(Ker(H^0(S,L) \to H^0(S,L \otimes \mathcal{O}_Z)))$.
The latter is a constant dimensional projective space from the assumption that $L$ is $3k-1$ very ample. Let $\mathcal{Z} \seq I_m \times S$ be the universal
closed subscheme associated to $I_m$ and consider the diagram
$$\begin{CD}
I_m \times S @>p>> S\\
@VqVV \\
I_m
\end{CD}$$ where $p$ and $q$ are the projections. Define $\mathcal{E}:= q_*(Ker(p^*L \to p^*L \otimes \mathcal{O}_{\mathcal{Z}}))$. Then
$\mathcal{E}$ is a vector bundle on $I_m$ and $\mathcal{N}_m$ is defined at $\textbf{Proj}(\text{Sym}(\mathcal{E}^*))$, where $\mathcal{E}^*$
is the dual bundle to $\mathcal{E}$. 
Since $I_m$ is irreducible, we see that $\mathcal{N}_m$ is also irreducible. Indeed, $\mathcal{N}_m \to I_m$ is a flat, surjective
morphism with irreducible fibres.

To simplify the notation, we define $V^{m}(S,kL)$ as the Severi variety of irreducible curves $C \seq S$ in the linear
system $|kL|$ with exactly $m$ nodes, so that $V_{1+k^2(g-1)-m}(S,kL)=V^{m}(S,kL)$. 
\begin{thm}
 Suppose $S$ is a projective surface and $L \in Pic(S)$ is $3k-1$ very ample. Then the Severi variety $V^{m}(S,L)$ of irreducible curves $C \seq S$ in the linear
system $|L|$ with exactly $m \leq k$ nodes and no other singularities is irreducible.
\end{thm}
\begin{proof}
In fact $V^{m}(S,L)$ is the open set of $\mathcal{N}_m$ parametrising irreducible curves with exactly $m$ nodes.
\end{proof}

The next proposition studies the $k$ amplitude of ample vector bundles on generic K3 surfaces of genus $g$. 
Note that the generic $K3$ surface $S$ has $Pic(S) \simeq \mathbb{Z}$. Let $L$ 
be an ample, primitive line bundle on $S$, then $(L \cdot L)=2g-2$. Any curve in the linear system $|kL|$ has arithmetic genus $p_a=1+k^2 (g-1)$.
\begin{prop}
 Let $S$ be a K3 surface with $Pic(S) \simeq \mathbb{Z}$, and let $L \in Pic(S)$ be an effective generator of $Pic(S)$. Suppose $(L \cdot L) \geq 4k$. Then $L$ is 
$k$ very ample.
\end{prop}
\begin{proof}
 This follows from the criterion of \cite[Thm 1.1]{knut}, since for any effective line bundle $D \in Pic(S)$ we have $D \sim m L$ for some $m \geq 1$ and so
$2(D \cdot D) > (D \cdot L)$.
\end{proof}
\begin{cor} \label{nonrat-cor}
  Let $S$ be a K3 surface of genus $g$ with $Pic(S) \simeq \mathbb{Z}$, and let $L \in Pic(S)$ be an effective and primitive line bundle. 
  Then $V_{h}(S,kL)$ is irreducible or empty
for $h \geq \frac{k(6k-1)(g-1)+4}{6}$. 
\end{cor}
\begin{proof}
Suppose $h \geq \frac{k(6k-1)(g-1)+4}{6}$. Let $n$ be the number of 
nodes of irreducible, nodal curves in the linear system $|kL|$ and let $p_a$ be the arithmetic
genus of the curves in the system. Then
\begin{align*}
n &\leq p_a-\frac{k(6k-1)(g-1)+4}{6} \\
&=\frac{6+6k^2(g-1)-k(6k-1)(g-1)-4}{6} \\
&=\frac{2+k(g-1)}{6}.
\end{align*}
On the other hand $L$ is $\frac{(L \cdot L)}{4}$ very ample so $kL$ is $\frac{k(L \cdot L)}{4}=\frac{k(g-1)}{2}$ very
ample. Hence $L$ is at least $3n-1$ very ample. 
\end{proof}
\begin{rem}
A stronger inequality is proven in \cite[Cor.\ 2.6]{keilin} in which irreducibility holds when the number of nodes is of the order $k^2(L \cdot L)$,
rather than $k(L \cdot L)$. The proof uses more sophisticated methods.
\end{rem}
\begin{cor}
 The moduli space $M_{h,k}^g$ of tuples $(C,S,L)$ with $(S,L)$ a primitively polarised K3 surface of genus $g$ and $C \in |kL|$ an
irreducible, nodal curve of genus $h$ is irreducible for $h \geq \frac{k(6k-1)(g-1)+4}{6}$.
\end{cor}
\begin{proof}
 We have a flat map $M_{h,k}^g \to W_g$ where $W_g$ is the moduli space of embedded, primitively polarised K3 surfaces. In fact, this map is even smooth, by the argument
 of \cite[Prop. 4.8]{flam}. By Corollary \ref{nonrat-cor} the generic 
fibre is irreducible so the result follows.
\end{proof}

This result is particularly interesting when $k=1$, i.e.\ we are looking at curves in a primitive class. 
In this case we have established irreducibility of $M_{h,1}^g$ in the range $\frac{5g-1}{6} \leq h \leq g$. 
So we have proven the result for top $\frac{1}{6}$ of the allowed values
for the geometric genus of the curves.

\bibliography{biblio}
\end{document}